\documentclass[oneside,english]{amsart}
\usepackage[T1]{fontenc}
\usepackage{amsthm}
\usepackage{amstext}
\usepackage{graphicx}
\usepackage{setspace}
\makeatletter

\numberwithin{equation}{section}
\numberwithin{figure}{section}
\theoremstyle{plain}
\newtheorem{thm}{Theorem}
  \theoremstyle{plain}
  \newtheorem{lem}{Lemma}
  \theoremstyle{plain}
  \newtheorem{cor}{Corollary}

\makeatother

\usepackage{babel}
\usepackage{fullpage}

\begin{document}

\title[Poincar\'{e} and Cheeger bounds for random walk on a graph]{A note on the Poincar\'{e} and Cheeger inequalities for simple random
walk on a connected graph}



\maketitle
\begin{center}
\author{John Pike\\ \small Department of Mathematics\\[-0.8ex] \small University of Southern California\\ \small \texttt{jpike@usc.edu}}
\par\end{center}

\begin{abstract}
In 1991, Persi Diaconis and Daniel Stroock obtained two canonical
path bounds on the second largest eigenvalue for simple random walk
on a connected graph, the Poincar\'{e} and Cheeger bounds, and they
raised the question as to whether the Poincar\'{e} bound is always
superior. In this paper, we present some background on these issues,
provide an example where Cheeger beats Poincar\'{e}, establish some
sufficient conditions on the canonical paths for the Poincar\'{e}
bound to triumph, and show that there is always a choice of paths
for which this happens.
\end{abstract}

\section{Background and Notation}

Let $G=G(X,E)$ be any simple, connected, undirected, and unweighted
graph with edge set $E$ and finite vertex set $X$. For each $(x,y)\in X\times X$
with $x\neq y$, choose a unique oriented path $\gamma_{x,y}$ from
$x$ to $y$. Let $\Gamma$ denote this collection of canonical paths
(also known as a routing) and define

\[
n=\left|X\right|,\; d=\max_{x\in X}\textrm{deg}(x),\;\gamma_{*}=\max_{\gamma\in\Gamma}\left|\gamma\right|,\;\text{and}\;\, b=\max_{\overrightarrow{e}\in\overrightarrow{E}}\left|\{\gamma\in\Gamma:\gamma\ni\overrightarrow{e}\}\right|.\]
Note that in the definition of $b$, the \textit{bottleneck number}
of $(G,\Gamma)$, the maximum is taken over directed edges and so
is to be distinguished from the related concept of the \textit{edge-forwarding
index} of $(G,\Gamma)$ (see \cite{XuXu}). We distinguish directed
and undirected edges by adorning the former with an arrow, and if
$e$ is an undirected edge connecting vertices $x$ and $y$, we write
$e=\{x,y\}$. Throughout this paper, the term \textit{path} refers
to a sequence of vertices where successive terms are connected by
an edge. Repeated vertices are allowed, but no edge may appear more
than once. (Some authors refer to such a path as a \textit{trail}.)
If $\overrightarrow{e}$ is an edge from $x$ to $y$, we write $\gamma\ni e$,
$\gamma\ni\overrightarrow{e}$ if the vertex sequence defining $\gamma$
contains $x$ and $y$ as consecutive terms, $x$ preceding $y$ in
the latter case.

A simple random walk on $G$ begins at some vertex $x_{0}$ and then
proceeds by moving to a neighboring vertex chosen uniformly at random.
This defines a Markov process $\{X_{k}\}$ with state space $X$,
transition probabilities \[
K(x,y)=\begin{cases}
\frac{1}{\textrm{deg}(x)}, & \{x,y\}\in E\\
0, & \textrm{otherwise}\end{cases},\]
and stationary distribution \[
\pi(x)=\frac{\textrm{deg}(x)}{2\left|E\right|}.\]
This Markov chain is irreducible and reversible, thus the operator
$K$ defined by \[
[K\phi](x)=\sum_{y\in X}K(x,y)\phi(y)\]
is a self-adjoint contraction on $L^{2}(\pi)$ with real eigenvalues
\[
1=\beta_{0}>\beta_{1}\geq\ldots\geq\beta_{n-1}\geq-1\]

\begin{flushleft}
whose corresponding eigenfunctions are orthogonal with respect to
the inner product \[
\left\langle \phi,\psi\right\rangle _{\pi}=\sum_{x\in X}\phi(x)\psi(x)\pi(x).\]

\par\end{flushleft}

\begin{flushleft}
It is of interest to estimate $\beta_{*}=\max\left\{ \beta_{1},\left|\beta_{n-1}\right|\right\} $
as this quantity can be used to bound the \textit{r}-step distance
to stationarity. For example, when $K$ is reversible with respect
to $\pi$, letting $K_{x}^{r}$ denote the distribution of $X_{r}$
given that $X_{0}=x$, we have the classical bound on the total variation
distance to stationarity \cite{DiaStr} \[
\left\Vert K_{x}^{r}-\pi\right\Vert _{TV}\leq\frac{1}{2}\beta_{*}^{r}\sqrt{\frac{1-\pi(x)}{\pi(x)}}.\]
More generally, consideration of the Jordan normal form of the transition
matrix shows that the exponential rate of convergence of any ergodic
Markov chain is governed by the second largest eigenvalue (in magnitude).
Because one can often ensure that $\beta_{1}\geq\left|\beta_{n-1}\right|$
- by adding holding probabilities, for example - much of the research
has focused on bounding $\beta_{1}$. 
\par\end{flushleft}

In the above setting, Diaconis and Stroock \cite{DiaStr} give the
Poincar\'{e} inequality \[
\beta_{1}\leq1-\frac{2\left|E\right|}{d^{2}\gamma_{*}b}\]
and the Cheeger inequality \[
\beta_{1}\leq1-\frac{\left|E\right|^{2}}{2d^{4}b^{2}}.\]
These inequalities are corollaries of results that hold for all irreducible,
reversible Markov chains and are based on geometric techniques (derived
by their namesakes) for bounding the spectral gap of the Laplacian
on a Riemannian manifold. Both ultimately rely on the variational
characterization of the eigenvalues of the discrete Laplacian $L=I-K$:
\[
1-\beta_{1}=\inf_{\phi\textrm{ nonconstant}}\frac{\mathcal{E}(\phi,\phi)}{\textrm{Var}_{\pi}(\phi)}\]
where $\mathcal{E}(\phi,\phi)$ is the Dirichlet form \[
\mathcal{E}(\phi,\phi)=\left\langle \phi,L\phi\right\rangle _{\pi}=\frac{1}{2}\sum_{x,y\in X}[\phi(x)-\phi(y)]^{2}\pi(x)K(x,y)\]
and \[
\textrm{Var}_{\pi}(\phi)=\left\langle \phi-E_{\pi}[\phi],\phi-E_{\pi}[\phi]\right\rangle _{\pi}=\frac{1}{2}\sum_{x,y\in X}[\phi(x)-\phi(y)]^{2}\pi(x)\pi(y)\]
is the variance of $\phi$ with respect to $\pi$, $E_{\pi}[\phi]=\left\langle \phi,1\right\rangle _{\pi}$
being the corresponding expectation. This characterization follows
easily from the Courant-Fischer theorem \cite{HoJo} and the properties
of $K$.

The use of canonical paths in this framework originated in the work
of Jerrum and Sinclair on approximating the permanent of a $0$-$1$
matrix \cite{JerSin}. For the purposes of this exposition, we will
only be considering the aforementioned inequalities for random walks
on graphs, but it should be noted that they are both overestimates
of the more general Poincar\'{e} and Cheeger bounds. Also, observe
that the variational characterization immediately gives lower bounds
on $\beta_{1}$ by evaluating the Rayleigh quotient at any nonconstant
$\phi$. The interested reader is encouraged to consult \cite{DiaStr}
for the general bounds, their proofs, and the derivation of these
particular cases for random walk on a graph. 

Though both inequalities are valid for any choice of $\Gamma$, their
utility hinges upon a clever selection of canonical paths. In particular,
one seeks to minimize $\gamma_{*}b$ for Poincar\'{e} and $b$ for
Cheeger. It was noted in \cite{DiaStr} that the Poincar\'{e} bound
is often superior, regardless of the choice of $\Gamma$, but it was
left as an open question whether this is always the case. A little
algebra shows that this is equivalent to asking whether $4d^{2}b\geq\gamma_{*}\left|E\right|$
for all choices of $\Gamma$. In addition to better understanding
eigenvalue estimates for simple random walk on connected graphs, knowledge
of the conditions under which the preceding inequality holds is of
interest in its own right. For example, bounds relating standard graph
theoretic quantities with measures of bottlenecking may be useful
in applications involving network management or optimal distribution.

Jason Fulman and Elizabeth Wilmer \cite{FulWil} have shown that Poincar\'{e}
beats Cheeger for simple random walk on trees, for which there is
only one choice of canonical paths. They also show that for any vertex
transitive graph, such as the Cayley graph of a group with a symmetric
generating set, and any collection of paths $\Gamma,$ one has the
inequality $bd^{2}\geq D\left|E\right|$ where $D=\textrm{diam}(G)$.
Thus Poincar\'{e} beats Cheeger for random walk on vertex transitive
graphs whenever $\gamma_{*}\leq4D$. In particular, this implies that
the Poincar\'{e} bound is superior in these graphs when the paths
are taken to be geodesics. 

For more background and examples, the reader is referred to \cite{DiaStr}
and \cite{FulWil}. Other useful references include \cite{JerSin,Sin,Chung}.
By way of a counterexample, we answer the question posed by Diaconis
and Stroock in the negative. Moreover, we extend the work of Fulman
and Wilmer by providing more general criteria for the Poincar\'{e}
bound to prevail.

\section{A Counterexample}

Consider the case when $G$ is a complete graph on $n$ vertices.
Then $d=n-1$ and $\left|E\right|=\binom{n}{2}$. Now let $\gamma$
be a Hamiltonian path with initial vertex $x_{0}$ and terminal vertex
$y_{0}$, and define a routing $\Gamma$ by letting $\gamma_{x,y}$
be the unique (length 1) oriented geodesic from $x$ to $y$ for each
$(x,y)\in X\times X\setminus\{(x_{0},y_{0})\}$ and letting $\gamma_{x_{0},y_{0}}=\gamma$.
Then $\gamma_{*}=\left|\gamma\right|=n-1$ and $b=2$. The conjectured
inequality is thus \[
8(n-1)^{2}=4d^{2}b\geq\gamma_{*}\left|E\right|=(n-1)\left(\frac{n^{2}-n}{2}\right).\]
Because this fails when $n\geq17$, there are infinitely many graphs
where Cheeger beats Poincar\'{e} for some choice of canonical paths,
hence the Poincar\'{e} bound is not uniformly superior to the Cheeger
bound. Observe that since the left hand side is $\Omega(n^{2})$ and
the right hand side is $\Omega(n^{3})$, the inequality cannot be
salvaged by adjusting constants. Also, this particular choice of paths
is kind of an overkill in that the argument actually shows that if
one path in $\Gamma$ has length greater than $16\frac{n-1}{n}$ and
the rest are geodesics, then we still get a counterexample. At the
other extreme, observe that when $n$ is odd, we can take the long
path, $\gamma$, to be an Eulerian cycle with an edge deleted, and
when $n$ is even, we can take $\gamma$ to be an Eulerian cycle on
a subgraph on $n-1$ vertices with an edge deleted. These constructions
show that counterexamples appear in all complete graphs on $n\geq7$
vertices.

Let us now examine the case of a complete graph on $n\geq3$ vertices
a little closer. The transition probabilities are given by \[
K(x,y)=\frac{1}{n-1}\left(1-\delta_{x}(y)\right),\]
so it is easy to see that $G$ has $1$ as a simple eigenvalue and
its only other eigenvalue is $-\frac{1}{n-1}$ with multiplicity $n-1$.
The preceding analysis shows that if we take $\Gamma$ as above, then
the Poincar\'{e} bound is \[
\beta_{1}\leq1-\frac{n}{2(n-1)^{2}}\]
and the Cheeger bound is \[
\beta_{1}\leq1-\frac{n^{2}}{32(n-1)^{2}}.\]
Neither bound is remotely sharp, but Cheeger certainly comes out on
top for $n$ sufficiently large. If we take $\Gamma$ to be the unique
set of geodesics, in which case $\gamma_{*}=b=1$, then we obtain
a Poincar\'{e} bound of \[
1-\frac{2\left|E\right|}{d^{2}\gamma_{*}b}=-\frac{1}{n-1}.\]
The corresponding Cheeger bound is \[
1-\frac{\left|E\right|^{2}}{2d^{4}b^{2}}=1-\frac{1}{8}\left(\frac{n}{n-1}\right)^{2}.\]
Because $\gamma_{*}$ and $b$ are small as they can be in this case,
these are the best possible bounds. 

Thus even though we have an example where Cheeger beats Poincar\'{e}
(on a vertex transitive graph), it involves a choice of paths that
yields terrible bounds while the best possible choice of paths gives
an ideal Poincar\'{e} bound that is much better than the optimal Cheeger
bound. Perhaps then the relevant question is whether the best possible
Poincar\'{e} bound is always better than the best possible Cheeger
bound. If the Poincar\'{e} v. Cheeger conjecture is stated in terms
of optimal bounds, then one can stipulate that the paths do not have
repeated vertices as allowing them to contain cycles can only increase
the bottleneck number and the longest path length. However, the counterexample
shows that this restriction alone does not guarantee that the Poincar\'{e}
bound is uniformly better than the Cheeger bound. Moreover, for an
arbitrary graph, there is no reason to suspect that the optimal Poincar\'{e}
bound is realized by a choice of paths which gives the optimal Cheeger
bound since a choice of paths which minimizes $\gamma_{*}b$ need
not minimize $b$ (and vice versa), so this question is more involved
than asking when $4d^{2}b\geq\gamma_{*}\left|E\right|$. Indeed, it
seems that a resolution of this issue would require some characterization
of the collections of canonical paths that yield the best bounds,
which is probably both the most important and the most difficult problem
associated with these path-based eigenvalue inequalities. For example,
if $\pi(G,\Gamma)$ denotes the edge-forwarding index of $(G,\Gamma)$,
which is defined just like the bottleneck number but in terms of undirected
edges, then it is known \cite{HMSO} that the problem of determining
whether a given integer upper-bounds the minimum of $\pi(G,\Gamma)$
over all minimal, symmetric, and consistent routings $\Gamma$ is
NP-complete when $\textrm{Diam}(G)\geq3$.

At this point, we observe that the construction of the counterexample
involved a choice of canonical paths with one exceptionally long path
and many shorter paths. This is no coincidence. For any choice of
paths, $\Gamma$, there are $M=\sum_{\gamma\in\Gamma}\left|\gamma\right|$
oriented edges counting multiplicity. Since there are a total of $2\left|E\right|$
oriented edges in the graph, the pigeonhole principle implies $b\geq\frac{M}{2\left|E\right|}$.
In terms of the average path length, $\bar{\gamma}=\frac{1}{n^{2}}\sum_{\gamma\in\Gamma}\left|\gamma\right|$,
we get the inequality $b\geq\frac{n^{2}\bar{\gamma}}{2\left|E\right|}$.
(Note that we are including the $n$ empty paths in this average.)

This shows that a sufficient condition for Poincar\'{e} to beat Cheeger
is that $2d^{2}n^{2}\bar{\gamma}\geq\gamma_{*}\left|E\right|^{2}$.
Because $2\left|E\right|=\sum_{x\in X}\textrm{deg}(x)\leq dn$, it
follows that Poincar\'{e} beats Cheeger whenever $8\bar{\gamma}\geq\gamma_{*}$.
This idea was used implicitly in \cite{FulWil}, but the author feels
that it is significant enough to be stated directly as
\begin{thm}
For any simple connected graph $G$, if $\Gamma$ is a set of canonical
paths that satisfies $8\bar{\gamma}\geq\gamma_{*}$, then \textup{$4d^{2}b\geq\gamma_{*}\left|E\right|$,
}hence the Poincar\'{e} bound is superior to the Cheeger bound for
this choice of paths. 
\end{thm}
Thus in order for the Cheeger inequality to prevail, the longest path
length must exceed the average path length considerably. When attention
is restricted to geodesic paths, this shows that Poincar\'{e} beats
Cheeger whenever the mean distance is at least one eighth of the diameter.
In addition to providing a nice general criterion, this observation
shows that Poincar\'{e} beats Cheeger along geodesics for a larger
class of graphs than just those which are vertex transitive.

Recalling the inequality $b\geq\frac{M}{2\left|E\right|}$ where $M=\sum_{\gamma\in\Gamma}\left|\gamma\right|$,
one sees that $M$ is minimized when $\Gamma$ is taken to be a collection
of geodesics. This choice also minimizes $\gamma_{*}$, though there
could be other choices of canonical paths not consisting entirely
of geodesics for which $\gamma_{*}=\textrm{Diam}(G)$. This suggests
that geodesic paths are often a good starting point for finding optimal
bounds. Of course $\frac{M}{2\left|E\right|}$ is only a lower bound
for $b$ and there are often many choices for $\Gamma$ such that
all paths are geodesics, so this is in no way a sufficient criterion
for optimization. Indeed, one can sometimes reduce $b$ without increasing
$\gamma_{*}$ by taking a slight detour in traveling between certain
vertices. A little thought will show that this is the case in the
following graph. 

\begin{center}
\includegraphics[width=5cm,height=3cm]{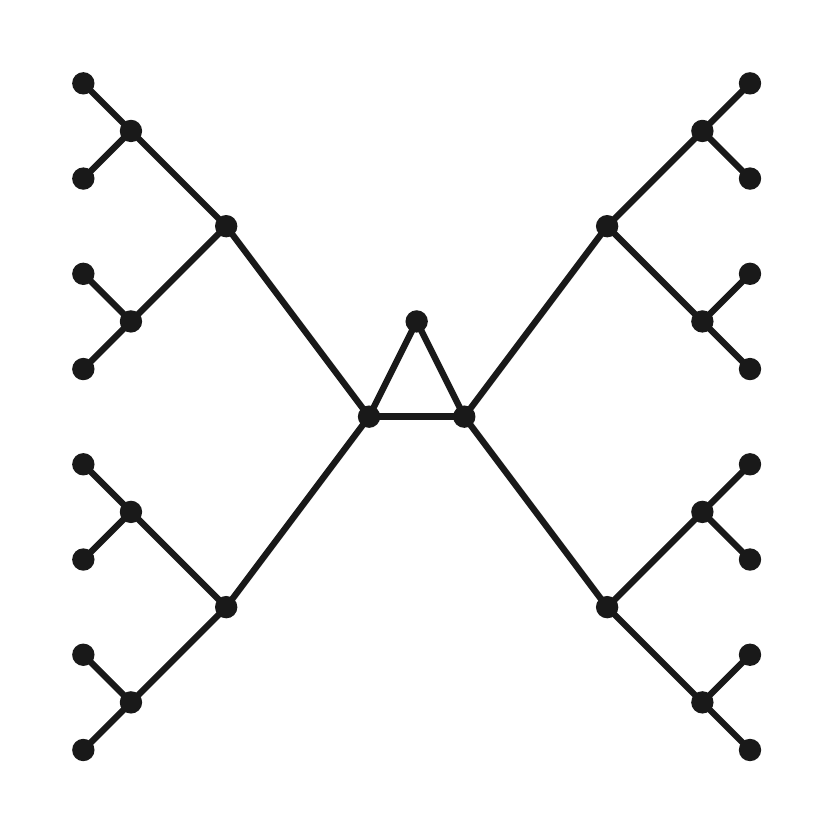}
\par\end{center}

Still, for the reasons indicated above and the fact that geodesics
are often among the more obvious choices for canonical paths, it would
be interesting to know more about when Poincar\'{e} beats Cheeger
along geodesics and when geodesic routings yield optimal bounds (as
in the case of complete graphs). All of the examples in \cite{DiaStr}
and \cite{FulWil} used geodesic paths, but the above graph shows
that this is not always the best choice.

\section{Spanning Trees}

In this section, we show that if $\Gamma$ is taken to be the set
of canonical paths along any spanning tree of any simple connected
graph, then the Poincar\'{e} bound is strictly better than the Cheeger
bound. Specifically, we have $d^{2}b\geq\gamma_{*}\left|E\right|$
for such $\Gamma$. Moreover, the inequality is strict whenever the
graph has more than two vertices, or equivalently, when $d\geq2$.
Since every connected graph has at least one spanning tree, this will
show that there is always a choice of paths for which Poincar\'{e}
beats Cheeger. 

To begin, let $G=G(X,E)$ be a simple connected graph on $\left|X\right|=n$
vertices and let $T$ be any spanning tree for $G$. Then we can define
$\Gamma=\Gamma(T)$ to be the (unique) set of paths along $T$. We
will show that Poincar\'{e} beats Cheeger for such routings by first
dealing with the case where $\gamma_{*}$ is large and then using
the results of \cite{FulWil} to handle the remaining case. We note
at the outset that as there is only one connected graph with $d=1$
and it is easily verified that $d^{2}b=1=\gamma_{*}\left|E\right|$
in this case, we can assume throughout that $d\geq2$. Now let us
say that that a routing is subordinate with respect to a path $\gamma_{x,y}=v_{0},v_{1},...,v_{m}$
if for all $0\leq i<j\leq m$ with $v_{i}\neq v_{j}$, $\gamma_{v_{i},v_{j}}=v_{i},v_{i+1},...,v_{j}$.
With this terminology, we can dispense with the large $\gamma_{*}$
case using the following lemma.
\begin{lem}
If $\Gamma$ is subordinate with respect to a longest path $\gamma$
and $\gamma_{*}=\left|\gamma\right|>\frac{4\left|E\right|}{d^{2}}-2$,
then $d^{2}b>\gamma_{*}\left|E\right|$.\end{lem}
\begin{proof}
\[
b\geq\left\lfloor \frac{\gamma_{*}+1}{2}\right\rfloor \left((\gamma_{*}+1)-\left\lfloor \frac{\gamma_{*}+1}{2}\right\rfloor \right)\geq\frac{\gamma_{*}(\gamma_{*}+2)}{4}\]
because if $\gamma=v_{0},v_{1},...,v_{\gamma_{*}}$, then for all
$0\leq i<j\leq\gamma_{*},$ $\gamma_{v_{i},v_{j}}=v_{i},v_{i+1},...,v_{j}$,
so there are $\left\lfloor \frac{\gamma_{*}+1}{2}\right\rfloor $
initial vertices and $(\gamma_{*}+1)-\left\lfloor \frac{\gamma_{*}+1}{2}\right\rfloor $
terminal vertices of paths through a central edge of $\gamma$. Consequently,
if $\gamma_{*}>\frac{4\left|E\right|}{d^{2}}-2$ (thus $d^{2}\gamma_{*}+2d^{2}>4\left|E\right|$),
then \[
d^{2}b\geq d^{2}\frac{\gamma_{*}(\gamma_{*}+2)}{4}=\gamma_{*}\frac{d^{2}\gamma_{*}+2d^{2}}{4}>\gamma_{*}\left|E\right|.\]

\end{proof}
Since a spanning tree routing is subordinate with respect to each
of its paths, Lemma 1 shows that if $\Gamma=\Gamma(T)$ is the set
of canonical paths along a spanning tree $T$ of $G,$ then the Poincar\'{e}
bound is strictly better than the Cheeger bound whenever $\gamma_{*}>\frac{4\left|E\right|}{d^{2}}-2$. 

We now note that for any connected graph $G$ on $n$ vertices, we
must have $\left|E\right|\geq n-1$. When $\left|E\right|=n-1$, $G$
is a tree and there is only one choice of canonical paths. In this
case, every path is subordinate with respect to the longest path,
which has length $\gamma_{*}=\left|E\right|=n-1$, so the first line
in the proof of Lemma 1 shows that \[
b\geq\frac{\gamma_{*}(\gamma_{*}+2)}{4}>\frac{1}{4}\gamma_{*}\left|E\right|.\]
Since $d\geq2$ whenever $n>2$, we see that if $G$ is a graph on
$n>2$ vertices with $\left|E\right|=n-1$, then $d^{2}b>\gamma_{*}\left|E\right|$.
Consequently, we can assume henceforth that $\left|E\right|\geq n$. 

For the remaining case, we appeal to Theorem 4 in \cite{FulWil},
the proof of which is included for completeness. 
\begin{lem}
Let $T$ be a tree with $n$ vertices having maximal degree $d_{T}\geq2$
and bottleneck number $b_{T}$. Then \[
b_{T}\geq\frac{(n-1)^{2}}{d_{T}^{2}}.\]
\end{lem}
\begin{proof}
Since $T$ is a tree with $n$ vertices, $T$ has $n-1$ edges, and
for any oriented edge $e$, the number of pairs of vertices $(x,y)$
with $e\in\gamma_{x,y}$ is of the form $k(n-k)$, $1\leq k\leq\frac{n}{2}$
where deletion of edge $e$ cuts $T$ into two components of size
$k$ and $n-k$, respectively. Because $k(n-k)$ is increasing in
$k$ for $1\leq k\leq\frac{n}{2}$, it is enough to show there exists
an edge $e^{*}$ whose removal divides $T$ into two components, the
smaller of which has size at least $\frac{n-1}{d_{T}}$. Given such
an edge, we have \[
b_{T}\geq\left(\frac{n-1}{d_{T}}\right)\left(n-\frac{n-1}{d_{T}}\right)\geq\left(\frac{(d_{T}-1)(n-1)^{2}}{d_{T}^{2}}\right)\geq\frac{(n-1)^{2}}{d_{T}^{2}}.\]

Thus we need only to demonstrate the existence of $e^{*}$. To this
end, note that if there is an edge which cuts $T$ into two pieces
of equal size, then we may take this edge to be $e^{*}$. Otherwise,
the deletion of any edge divides $T$ into two pieces of size $k$
and $n-k$, respectively, where $k<n-k$. In this case, we can orient
the edges by directing each edge from its endpoint in the component
of size $k$ to its endpoint in the component of size $n-k$. Because
there are $n$ vertices and $n-1$ edges, there must exist a vertex
$v^{*}$ with indegree $0$. Since $\textrm{deg}(v^{*})\leq d_{T}$,
there must be some edge leaving $v^{*}$ whose deletion cuts $T$
into two components, the smaller of which has size at least $\frac{n-1}{d_{T}}$.
Calling this edge $e^{*}$ establishes the result.
\end{proof}
Now by Lemma 1, we can assume that $\gamma_{*}\leq\frac{4\left|E\right|}{d^{2}}-2$.
Also, since $d\geq2$ implies there are at least 3 vertices in $G$
(and thus in $T$), the maximal degree of the vertices in $T$ is
$d_{T}\geq2$. Because $d\geq d_{T}$, Lemma 2 shows that \[
b=b_{T}\geq\frac{(n-1)^{2}}{d_{T}^{2}}\geq\frac{(n-1)^{2}}{d^{2}}.\]
Therefore, since $2\left|E\right|=\sum_{x\in X}deg(x)$, thus $2\left|E\right|\leq nd$,
and we may assume also that $\left|E\right|\geq n$, we have \[
\gamma_{*}\left|E\right|\leq\left(\frac{4\left|E\right|}{d^{2}}-2\right)\left|E\right|=\frac{(2\left|E\right|)^{2}}{d^{2}}-2\left|E\right|\leq n^{2}-2n<d^{2}\frac{(n-1)^{2}}{d^{2}}\leq d^{2}b.\]

We record the above result as
\begin{thm}
For any simple connected graph $G$, if $\Gamma$ is taken to be the
set of paths along any spanning tree of $G$, then $d^{2}b\geq\gamma_{*}\left|E\right|$
for this choice of $\Gamma$. If $G$ has at least 3 vertices, then
the inequality is strict.
\end{thm}
Since every connected graph has at least one spanning tree, we have
the immediate corollary
\begin{cor}
For every simple connected graph, there is a choice of canonical paths
such that the Poincar\'{e} bound is strictly better than the Cheeger
bound.
\end{cor}
It is worth pointing out that since trees are acyclic, these results
also do not depend on whether repeated vertices are allowed. Of course,
the above construction of $\Gamma$ is certainly not optimal as it
inflates both the bottleneck number and the longest path length by
precluding certain combinations of paths. Moreover, different choices
of spanning trees can yield different bounds. For instance, in the
complete graph on $n=2m+1\geq3$ vertices, two possible choices for
$T$ are a Hamiltonian path or the geodesics emanating from some vertex
$x_{0}$. In the first case, we get the Poincar\'{e} bound \[
\beta_{1}\leq1-\frac{4n}{(n-1)^{3}(n+1)}\]
and the Cheeger bound \[
\beta_{1}\leq1-\frac{2n^{2}}{(n-1)^{4}(n+1)^{2}},\]

\begin{flushleft}
while in the second case, the Poincar\'{e} bound is \[
\beta_{1}\leq1-\frac{n}{2(n-1)^{2}}\]
and the Cheeger bound is \[
\beta_{1}\leq1-\frac{n^{2}}{8(n-1)^{4}}.\]

\par\end{flushleft}

\section{Concluding Remarks}

We have established that Poincar\'{e} is not uniformly superior to
Cheeger for all simple connected graphs and Cheeger is not uniformly
superior to Poincar\'{e} for any simple connected graph. Strictly
speaking, this resolves the question put forth by Diaconis and Stroock,
but it would be nice to know whether the best possible Poincar\'{e}
bounds are always better than the best possible Cheeger bounds. A
more complete characterization of the choice of canonical paths for
which Poincar\'{e} beats Cheeger would likely prove helpful in settling
this question. Lemma 1 and Theorems 1 and 2 offer partial results
in this direction and subsume all previous findings, but the matter
is still far from being resolved. Perhaps examining the case of geodesic
routings more closely would shed some light on these issues. Also,
since every connected graph can be obtained by adding edges to a spanning
tree or deleting edges from a complete graph, it may be possible to
obtain related results using surgery methods in conjunction with the
above analyses of these extreme cases. Lastly, and probably of greatest
practical importance, one would like to know more about how to find
routings which yield the best bounds.

\section*{Acknowledgement}

The author would like to thank Jason Fulman for introducing him to
the ideas discussed and for his encouragement and helpful comments
in writing this paper.


\begin{thebibliography}{8}

\bibitem{Chung}
\textsc{Chung, F.R.K.} (1997) Spectral graph theory. CBMS
Regional Conference Series in Mathematics, 92. \textit{Published for
the Conference Board of the Mathematical Sciences, Washington, DC;
by the American Mathematical Society, Providence, RI}.

\bibitem{DiaStr}
\textsc{Diaconis, P. and Stroock, D.} (1991) Geometric bounds
for eigenvalues of Markov chains. \textit{Ann. Appl. Probab.} \textbf{1},
no. 1, 36--61.

\bibitem{FulWil}
\textsc{Fulman, J. and Wilmer, E.L.} (1999) Comparing eigenvalue
bounds for Markov chains: when does Poincar\'{e} beat Cheeger? \textit{Ann.
Appl. Probab.} \textbf{9}, no. 1, 1--13.

\bibitem{HMSO}
\textsc{Heydemann, M.-C., Meyer, J.-C., Sotteau, D. and Opatrn\'{y},
J.} (1994) Forwarding indices of consistent routings and their complexity.
\textit{Networks} \textbf{24}, no. 2, 75--82.

\bibitem{HoJo}
\textsc{Horn, R.A. and Johnson, C.R.} (1990) Matrix analysis.
Corrected reprint of the 1985 original. \textit{Cambridge University
Press, Cambridge}.

\bibitem{JerSin}
\textsc{Jerrum, M. and Sinclair, A.} (1989) Approximating
the permanent. \textit{SIAM J. Comput.} \textbf{18}, no. 6, 1149--1178.

\bibitem{Sin}
\textsc{Sinclair, A.} (1992) Improved bounds for mixing rates
of Markov chains and multicommodity flow. \textit{Combin. Probab.
Comput.} \textbf{1}, no. 4, 351--370.

\bibitem{XuXu}
\textsc{Xu, J.-M. and Xu, M.} The forwarding indices of graphs
- a survey. arXiv:1204.2604v1 [math.CO].

\end{thebibliography}
\end{document}